\documentclass[a4paper,reqno,11pt]{amsart}

\usepackage{amsmath}           
\usepackage{amssymb}            
\usepackage{mathabx}
\usepackage{latexsym}          
\usepackage{mathrsfs}          
\usepackage{eucal}             
\usepackage{verbatim}

\usepackage[all,2cell]{xy}
\usepackage{float}
\usepackage[colorlinks, linkcolor=blue,anchorcolor=Periwinkle,
citecolor=red,urlcolor=Emerald]{hyperref}
\usepackage[usenames,dvipsnames]{xcolor}
\usepackage{enumitem}

\usepackage{geometry,array} 
\usepackage{graphicx}
\usepackage{subfigure}
\usepackage{bookmark}
\usepackage{cleveref}

\usepackage{tikz-cd}  
\usepackage{adjustbox} 
\usepackage{dynkin-diagrams}
\usepackage{color}

\usepackage{tikz}\usetikzlibrary{matrix}
\tikzcdset{arrow style=tikz, diagrams={>=stealth}}

\usetikzlibrary{decorations.markings}
\tikzset{->-/.style={decoration={  markings,  mark=at position #1 with
			{\arrow{>}}},postaction={decorate}}}
\tikzset{-<-/.style={decoration={  markings,  mark=at position #1 with
			{\arrow{<}}},postaction={decorate}}}
\usepackage[all]{xy}
\usetikzlibrary{arrows}
\usetikzlibrary{decorations.pathreplacing,decorations.pathmorphing,shadings,fadings,calc}
\usepackage{extarrows}

\newtheorem{theorem}{Theorem}[section]
\newtheorem{lemma}[theorem]{Lemma}
\newtheorem{corollary}[theorem]{Corollary}
\newtheorem{example}[theorem]{Example}
\newtheorem{proposition}[theorem]{Proposition}

\theoremstyle{remark}
\newtheorem{remark}[theorem]{Remark}
\newtheorem{definition}[theorem]{Definition}
\numberwithin{equation}{section}

\def\add{\operatorname{add}}

\def\Hom{\operatorname{Hom}}
\def\End{\operatorname{End}}
\def\Ext{\operatorname{Ext}}

\def\thick{\operatorname{thick}}

\def\dim{\operatorname{dim}}

\newcommand{\im}{\operatorname{Im}}

\def\U{\mathcal{U}}
\def\V{\mathcal{V}}
\def\A{\mathcal{A}}
\def\T{\mathcal{T}}
\def\F{\mathcal{F}}
\def\B{\mathcal{B}}
\def\H{\mathcal{H}}
\def\D{\mathcal{D}}
\def\db #1{D^b( #1)}

\def\S{\mathbb{S}}
\renewcommand{\k}{\mathbf{k}}
\newcommand{\W}{\mathcal{W}}
\newcommand{\C}{\mathcal{C}}
\renewcommand{\mod}{\operatorname{mod}}

\newcommand{\lten}{\overset{\mathbf{L}}{\otimes}}
\newcommand{\rhom}{\mathbf{R}\mathrm{Hom}}
\newcommand{\lsup}[2]{{\vphantom{{#2}}}^{#1}{#2}}
\newcommand{\ps}{ \lsup{\perp}{\langle S\rangle}}
\title{Serre functor and torsion pairs}
\author{Zhe Han and Ping He$^\ast$}
\dedicatory{Dedicated to Professor Bin Zhu on the occasion of his 60th birthday.}
\address{Hz: School of Mathematics and Statistics
	Henan University 475004 Kaifeng, China}
\email{zhehan@vip.henu.edu.cn}
\address{Hp: Yanqi Lake Beijing Institute of Mathematical Sciences and Applications, 101408 Beijing, China}
\email{pinghe@bimsa.cn}
\date{\today}
\thanks{$^*$ the correspondence author}


\begin{document}
\begin{abstract}
	Given a torsion pair $(\T,\F)$ in an abelian category $\A$ and its Happel-Reiten-Smal{\o} tilt $\B$, the equivalence of the realization functor $\db\B\to \db\A$ is determined by some properties of the torsion pair \cite{CHZ}. We call $(\T,\F)$ satisfying such a property effaceable. If $\A$ is an Ext-finite abelian category with Serre duality, we prove that $(\T,\F)$ is effaceable implies that $\U_{\T}$ is closed under Serre functor. Conversely, when $\A$ is the module category of a finite-dimensional hereditary algebra, we prove that
	the torsion pair $(\T,\F)$ is effaceable if and only if $\U_\T$ is closed under the Serre functor via a recollement of $\db{\A}$. 
\end{abstract}
	
\maketitle


\section{Introduction}

Given an abelian category $\A$ with a torsion pair $(\T,\F)$, the heart $\B$ of the t-structure $(\U_\T,\V_\T)$ on $\db{\A}$ defined by the torsion pair is an abelian category \cite{HRS}. We call the category $\B$ Happel-Reiten-Smal{\o} tilt (HRS-tilt) of $\A$ with respect to $(\T,\F)$. By \cite{BBD}, there is a realization functor $\Phi\colon \db\A\to \db \B$ which extends the embedding $\B\to \db\A$. In general, the realization functor $\Phi$ is not an equivalence. There are some criteria on $(\T,\F)$ for the functor $\Phi$ being an equivalence \cite{CHZ}.
Given a t-structure on a triangulated category $\D$, the Serre functor plays an important role in the relation between the derived category of the heart and $\D$ \cite{SR16}.  In this paper, we consider how the aisle $\U_\T$ closed under the Serre functor relates to the property of $(\T,\F)$.

We call a torsion pair $(\mathcal T,\mathcal F)$ of $\mathcal A$ is \emph{effaceable} if for each object $X$ in $\mathcal A$, there is an exact sequence 
$0\to F_1\to F_2\to X\to T_1\to T_2\to 0$ with $T_i\in \T$ and $F_i\in \F$ such that the corresponding class in $\Ext^3_{\A}(T_2,F_1)$ vanishes. We shall show that this definition is equivalent to the existence of the epimorphism $C\to F[1]$ in $\B$ such that the composition $C\to F[1]\to T[2]$ is zero for any $T\in \T$ and $F\in \F$. According to \cite[Theorem~A]{CHZ}, the realization functor $\Phi$ is an equivalence if and only if $(\T,\F)$ is an effaceable torsion pair. Assume that the derived category $\db\A$ admits a Serre functor $\mathbb S$. The functor $\Phi$ being an equivalence is closely related to the aisle $\U$ satisfying $\S\U\subseteq \U$. Actually, they are equivalent for the module category $\A$ of a finite-dimensional hereditary algebra \cite{SR16}.

In this paper, we consider how the Serre functor affects the equivalence of the realization functor $\Phi\colon \db{\B}\to \db{\A}$ for HRS-tilt $\B$. For finitely generated  torsion pair $(\T,\F)$, the realization functor $\Phi $ is an equivalence if and only if the aisle $\U_\T$ satisfying that $\S\U_\T\subseteq \U_{\T}$, see \cite[Lemma 4.6]{LVY},\cite[Proposition 5.1]{SR16}. We will prove one direction for any torsion pair of $\A$ with Serre duality.  If $\A$ is a hereditary category, the other direction of the statement is proved via recollement of $\db{\A}$ associated with some exceptional objects. We also show how HRS-tilts are compatible with the recollement, compared to \cite[Section 6]{LVY}.

For a Hom-finite $\k$-linear triangulated category $\D$ with a Serre functor $\S$, and a bounded t-structure $(\D^{\leq 0}, \D^{\geq 0})$ on $\D$ with heart $\H$, we ask whether the following statement holds: a realization functor $\db{\H}\to\D$ is an equivalence if and only if $\D^{\leq 0}$ satisfies $\S \D^{\leq 0}\subseteq \D^{\leq 0}$. The statement holds for the derived category of a hereditary algebra \cite{SR16}. In \cite{S}, the author proved that this statement holds for the derived category of coherent sheaves over a weighted projective line of domestic or tubular type. In general, a bounded t-structure $(\U,\V)$  of $\db{\A}$ that is closed under the Serre functor does not imply there exists a triangle equivalence $\db{\H}\to \db{\A}$, \cite[Section 10]{SR16}. It is interesting to ask:
which kind of the t-structure $(\D^{\leq 0},\D^{\geq 0})$ on the category $\D$ such that the above statement holds. In \cite[Lemma 3.6]{LVY} and \cite[Proposition 5.1]{SR16}, the authors proved that the statement holds for t-structures generated by silting complexes.
We expect to prove the statement holds for any bounded t-structure $(\U,\V)$ of $\db{\A}$ induced by a torsion pair $(\T,\F)$. In Theorem \ref{thm:suff}, we prove one direction for any Ext-finite abelian category with Serre duality. 

\begin{theorem}(Theorem \ref{thm:suffr})\label{thm:suff}
	Let $\A$ be an Ext-finite abelian category with Serre duality. If a torsion pair $(\mathcal T,\mathcal F)$ on $\A$ is effaceable. Then the aisle $\U_\T$ is closed under the Serre functor.
\end{theorem}
We conjecture that if $U_{\T}$ is closed under Serre functor, then the torsion pair $(\mathcal T,\mathcal F)$ on $\A$ is effaceable. We could only prove this for $\A=\mod \Lambda$ where $\Lambda$ is a hereditary algebra based on the recollement of $\db{\A}$ associated to some exceptional object. 

\begin{theorem}\label{thm:main}
Let $\Lambda$ be a finite-dimensional hereditary algebra and $(\T,\F)$ be a torsion pair in $\A=\mod\Lambda$.  Let $\S$ be the Serre functor of  $\db{\A}$. Then the following statements are equivalent.
\begin{enumerate}
    \item The aisle $\U_\T$ of $\db{\A}$ determined by $\T$ satisfies that $\S\U_{\T}\subseteq \U_\T$.
    \item The torsion pair $(\mathcal T,\F)$ is effaceable.
\end{enumerate}
\end{theorem}

In Section \ref{sec:two}, we recall basic facts about tilting theory and perpendicular categories defined by exceptional objects. The recollement associated with exceptional objects plays an essential role in proving Theorem \ref{thm:main}. In Section \ref{sec:three}, we prove torsion pair $(\T,\F)$ is effaceable implies that the aisle $\U_\T$ is closed under the Serre functor. In Section \ref{sec:sp}, we prove the converse holds if either $\T$ is finitely generated or without any Ext-projective objects, and the factorization of the second extension between two objects in HRS-tilt. In Section \ref{sec:five}, we show how HRS-tilt interplays with the recollement and reduce the general case of $\T$ to the special case $\T$ without any Ext-projective objects.

\section{Preliminary}\label{sec:two}
For any subcategory $\C$ of a triangulated category $\D$, let $\add\C$ be the additive closure of $\C$. For any pair of subcategories $(\C_1,\C_2)$, we write $\C_1*\C_2$ the full subcategory of $\D$ consisting of those objects $Z$ such that there is a triangle $C_1\to Z\to C_2\to C_1[1]$ in $\D$, where $C_1\in\C_1$ and $C_2\in\C_2$. For an abelian category $\A$ and a pair of subcategories $\A_1,\A_2$, $\A_1\ast \A_2$ is the full subcategory consisting of objects 
\[\{X\in \A\mid  0\to A_1\to A\to A_2\to 0, \; A_1\in \A_1,\; A_2\in \A_2\}.\]
  
We denote by $\mod\Lambda$ the category of finite-dimensional right $\Lambda$-modules over a finite-dimensional algebra $\Lambda$.

\subsection{Tilting theory and t-structures}

Recall that a pair $(\T,\F)$ of full subcategories of $\A$ is called a \emph{torsion pair} if $\Hom_\A(T,F)=0$ for any $T\in\T$ and $F\in\F$, and $\A= \T\ast \F$. The subcategory $\T$ is called \emph{torsion class} and the subcategory $\F$ is called a \emph{torsion free class}.

For an abelian category $\A$ and a subcategory $\mathcal C$, an object $M\in\mathcal C$ is called an \emph{Ext-projective} in $\mathcal C$ if $\Ext^i_{\mathcal C}(M,C)=0$ for each $C\in\mathcal C$. A torsion class $\T$ is called \emph{finitely generated} if each object is a quotient of an object in $\add E$ for an Ext-projective object $E$ in $\T$.
We call a torsion pair $(\mathcal T,\mathcal F)$ of $\mathcal A$ is \emph{effaceable} if for each object $X$ in $\mathcal A$, there is an exact sequence 
\begin{equation}\label{eq:seq}
	0\to F_1\to F_2\to X\to T_1\to T_2\to 0
\end{equation}
with $T_i\in \T$ and $F_i\in \F, i=1,2$ such that the corresponding class in $\Ext^3_{\A}(T_2,F_1)$ vanishes. There are some equivalent characterizations of the effaceable torsion pair involving a tilted category, \cite{CHZ}.

\begin{definition}
    A pair $(\U,\V)$ of full subcategories of $\D$ is called a \emph{t-structure} (resp. \emph{co-t-structure}) if 
    \begin{itemize}
    	\item $\U[1]\subseteq\U$ and $\V[-1]\subseteq\V$ (resp. $\U[-1]\subseteq\U$ and $\V[1]\subseteq\V$);
    	\item $\Hom_{\db{\A}}(\U,\V[-1])=0$, and
    	\item for any $X\in\D$ there is a triangle 
    	\[X_\U\to X\to X_{\V[-1]}\to X_\U[1]\]
    	with $X_\U\in\U,X_{\V[-1]}\in\V[-1]$. 
    \end{itemize}
     We call subcategories $\U$ and $\V$ an \emph{aisle} and \emph{coaisle}, respectively. A t-structure $(\U,\V)$ is said to be \emph{bounded} if $\bigcup_{n\in\mathbb{Z}}\U[n]=\D=\bigcup_{n\in\mathbb{Z}}\V[n]$. The \emph{heart} of a t-structure $(\U,\V)$ is the full subcategory $\H=\U\cap \V$.
\end{definition} 

Note that $(\U[n],\V[n]),n\in\mathbb{Z}$ is still a t-structure. For any $X\in\A$, we use $\tau^\H_{\leq n}X$ (resp. $\tau^\H_{\geq n}X$) to denote $X_{\U[-n]}$ (resp. $X_{\V[-n]}$). It is not hard to see that $(\tau^\H_{\leq n}X)[n]=\tau^\H_{\leq0}(X[n])$. The n-homology of $X$ with respect to $\H$ is defined as 
\[H^n_\H(X):=(\tau^\H_{\geq n}\tau^\H_{\leq n}X)[n]\in\H.\] 
When $\D=\db\A$ and $\H$ is taken as the canonical heart $\A$, we omit the subscript and simply write $H^n$ for $H^n_\A$.

\begin{remark}
	co-t-structure are introduced by Pauksztello in \cite{Pa}, weight structure were introduced by Bondarko in \cite{Bo}. They are the same concept for triangulated categories. 
\end{remark}
A bounded t-structure is determined by its heart $\H$,
\[\U=\bigcup_{n\geq 0}\H[n]\ast \cdots \H[1]\ast \H,\quad \V=\bigcup_{n\geq 0} \H\ast\H[1]\ast \cdots \ast \H[-n].\]
Given a triangle functor $F\colon \C\to \D$ between triangulated categories $\C$ and $\D$ with t-structures $(\U_i,\V_i)$ for $i=1,2$, $F$ is called \emph{t-exact} if $F(\U_1)\subseteq \U_2$ and $F(\V_1)\subseteq \V_2$. 

For a bounded t-structure $(\U,\V)$ on a triangulated category $\D$ with heart   $\A$, a \emph{realization functor} is a functor $\Phi\colon \db{\A}\to \D$ which is t-exact and the restriction on $\A$ is identity. In \cite{BBD}, the existence of the realization functor depends on filtered structures over $\mathcal D$. Note that a realization functor exists if $\mathcal D$ is algebraic \cite{KV}. 

For each torsion pair $(\T,\F)$ in $\A$, let
\[\U_{\T}=\{X\in\db\A\mid H^0(X)\in\T,H^n(X)=0\text{ for any } n>0\}\]
and
\[\V_{\T}=\{Y\in\db\A\mid H^{-1}(Y)\in\F, H^n(Y)=0 \text{ for any } n<-1\}.\]
By \cite[Proposition~2.1]{HRS}, $(\U_{\T},\V_{\T})$ is a bounded t-structure in $\db{\A}$. The heart $\B=\F[1]\ast\T$ of $(\U,\V)$ is called the \emph{HRS-tilt} of $\A$. Note that the realization functor $F\colon \db{\B}\to \db \A$ exists for the t-structure $(\U_\T,\V_\T)$ given by a torsion pair $(\T,\F)$ on $\A$.  For more details about tilting theory and the realization functor, refer to \cite{PV}.

Conversely, given a bounded t-structure $(\mathcal{D}^{\leq 0}, \mathcal{D}^{\geq 0})$ on a triangulated category $\D$ with heart $\mathcal H$, each t-structure $(\U,\V)$ with $\D^{\leq0}[1]\subseteq\U\subseteq\D^{\leq 0}$ determines a torsion class $\T:=H^0_{\mathcal H}(\U)$ in $\mathcal H$. The following lemma shows that such a correspondence is a bijection.
\begin{proposition}\label{prop:int}{\cite[Lemma 1.1.2]{P}\cite[Proposition~2.3]{Woo}}
Let $\H$ be the heart of a bounded t-structure $(\mathcal{D}^{\leq 0}, \mathcal{D}^{\geq 0})$ on a triangulated category $\D$. Then there is a bijection between the set of all torsion classes $\T$ of $\H$ and the set of t-structures $(\U,\V)$ of $\D$ satisfying
\[\D^{\leq 0}[1]\subseteq \U\subseteq \D^{\leq 0},\]
given by $\T\mapsto \U_{\T}$.
\end{proposition}

\subsection{Perpendicular categories}
For a collection $\mathcal S$ of objects in an abelian category $\A$, the  categories  $\mathcal S^{\perp}$  
\emph{left perpendicular} to $\mathcal S$ are  the full subcategories consisting of all objects $X\in \A$ satisfying
\[\Hom(X,S)=0,\quad \Ext^1(X,S)=0, \quad \forall S\in \mathcal S.\]
The right perpendicular $^{\perp}\mathcal S$ to $\mathcal S$ is defined similarly. If $\A$ is hereditary, then both $\mathcal S^{\perp}$ and $^{\perp}\mathcal S$ are abelian and hereditary \cite[Proposition 1.1]{GL}. We mention that in \cite{J} Jasso defined \emph{$\tau$-perpendicular category}  associated to $U\in\mod\Lambda$ which is the full subcategory of $\mod(\Lambda)$ whose
objects are all $X\in \mod(\Lambda)$ such that
\[\Hom(U,X)=0,\quad \Hom(X,\tau U)=0. \]
If $U$ is a partial tilting module, then the $\tau$-perpendicular category associated to $U$ is exactly the right perpendicular category $\langle U\rangle^{\perp}$.

Let $\mathcal D$ be a triangulated category and $\mathcal S$ a collection of objects in $\D$, the \emph{right perpendicular category} $\mathcal S ^{\perp}$ of $\mathcal S$ is the full subcategory consisting of those objects
\[\{X\in \mathcal D\mid \Hom(Y,X[n])=0, \forall Y\in \mathcal S \text{ and } n\in\mathbb{Z}\}.\]
The \emph{left perpendicular category} $^{\perp}\mathcal S$ is defined similarly. A full triangulated subcategory of $\D$ is called \emph{thick} if it is closed under taking direct summands. For any collection of objects  $\mathcal S$ in $\mathcal D$, we write $\langle \mathcal S \rangle$ for the thick subcategory generated by $\mathcal S$. 
Both $ \mathcal S ^{\perp}$ and $^{\perp}\mathcal S$ are thick subcategories of $\D$. Moreover, $\mathcal S ^{\perp}\cong \langle\mathcal S\rangle ^{\perp}$ and $^{\perp}\mathcal S\cong \ps$. 

\subsection{Exceptional objects and recollement}

\begin{definition}
	A recollement  $(\mathcal C_1,\mathcal C,\mathcal C_2)$ of a triangulated category by another two triangulated categories is a diagram of six
	functors between these categories given by the following six functors:
\begin{equation}\label{eq:rec}
	\begin{tikzcd}
	\C_1
		\arrow[rr, "i_*=i_!"{description}] 
		&& \C
		\arrow[ll,  bend left=30, "i^!"{description}] 
		\arrow[ll, bend right=30, "i^*_{S}"{description}]
		\arrow[rr, "j^*=j^!"{description}]
		& &\C_2
		\arrow[ll, bend left=30, "j_*"{description}] 
		\arrow[ll, bend right=30, "j_!"{description}]
	\end{tikzcd}
\end{equation}		
	
	such that
	\begin{enumerate}
		\item $(i^*, i_*), (i_!, i^!), (j_!, j^*), (j^*, j_*)$ are adjoint pairs;
		\item$i_*, j_*, j_!$ are full embeddings;
		\item $i^! \circ j_* = 0$ ( thus $j^! \circ i_! = 0$ and $i^* \circ j_! = 0$);
		\item For each $C \in  \mathcal C$ there are triangles
		\[i_!i^!(C)\to  C \to j_*j^*(C) \to i_!i^!(C)[1]\]
		\[j_!j^!(C)\to  C \to i_*i^*(C)\to j_!j^!(C)[1]\]
		where the maps are given by adjunctions.
	\end{enumerate}
\end{definition}

We mention that there are extensive studies on the connections between recollement of derived categories and tilting theory, refer to \cite{AKL11}, \cite{LVY}, \cite{LV}, \cite{PV}.

An \emph{exceptional object} in a $\k$-linear triangulated category $\C$ (of finite type) is an object $\Hom_\C(E,E[n])=0$ for $n\neq 0$ and $\End(E)$ is a division ring. A sequence $(E_1,E_2,\ldots,E_n)$ of exceptional objects in $\C$ is an \emph{exceptional sequence} if
\[\Hom(E_j,E_i[\mathbb{Z}])=0,\; 1\leq i<j\leq n.\]

Let $\mathcal A$ be an Ext-finite abelian category of finite global dimension and $S\in\db{\A}$ an exceptional object.
Assume that $B=\End S$, then the functor $-\lten_BS\colon \db{\mod B}\to \db{\A}$ has a right adjoint $\rhom(S,-)$ and a left adjoint $\rhom(-,S)^*$. 

Recall that a subcategory of a triangulated category is called \emph{admissible} if the inclusion functor has both right and left adjoints.
\begin{lemma}\cite{B89}
	The subcategory $\langle S\rangle$ of $\db{\A}$ generated by an exceptional object $S$ is admissible and
	is equivalent to the derived category of vector spaces $\db{\k}$.
\end{lemma}

 The thick subcategory $\langle S\rangle$ generated by $S$ in $\db{\A}$ is
equivalent to $\db{ B}$ which is induced by $-\lten_BS$. The embedding 
$i\colon \langle S\rangle\to \db{\A}$ has a right and left adjoint, given on objects by 
\[i^!_S(X)=\rhom(S,X)\lten_B S,\; i^*_S(X)=\rhom(X,S)^*\lten_B S, \]see \cite[Theorem~3.2]{B89} and \cite[Section 2.4]{SR16}.   By the assumption of $S$, we have that \[i^!_S(X)\cong \oplus_{i\in\mathbb Z}\Hom(S,X[i])\otimes_\k S[-i], \quad i^*_S(X)=\oplus_{i\in\mathbb Z} \Hom(X,S[i])^*\otimes_\k S[i].\]


The above functors gives rise to adjoints of the embedding $^{\perp} S\to \db{\A}$  and $S^{\perp}\to \db{\A}$.  The right adjoint $T^*_S\colon \db{\A}\to {}^{\perp}S $ and left adjoint $T_S\colon \db{\A} \to S^{\perp}$, are defined respectively by the following triangles
\[T^*_S(X)\to X\xrightarrow{\alpha_X}  ii^*_S(X) \to T^*_S(X)[1],\]
\[T_S(X)[-1]\to i i^!_S(X)\xrightarrow{\beta_X} X\to T_S(X).\]
The functors $T_S(=j_*j^*)$ and $T_S^*(=j_!j^*)$ are called the \emph{left} and the \emph{right mutation} functors. The left and right mutation functors vanish on $\langle S\rangle$ and induce mutually inverse equivalences \cite{B89},
\[T_S\colon \ps\to \langle S\rangle ^{\perp},\; T^*_S\colon \langle S\rangle^{\perp}\to \ps.\]

The addmissible subcategory $\langle S\rangle$ gives rise to a recollement as follows, see \cite[Proposition~2.8.2]{Rom}.
\begin{lemma}\label{lem:recl}
For an exceptional object $S\in \db{\A}$, there is a recollement
	\begin{equation}\label{eq:recl}  \begin{tikzcd}
	\langle S\rangle
	\arrow[rr, "i_*"{description}] 
	&& \db{\A}
	\arrow[ll, bend left=30, "i^!_S"{description}] 
	\arrow[ll, bend right=30, "i^*_{S}"{description}]
	\arrow[rr, "j^*"{description}]
	& &\lsup{\perp}{\langle S\rangle}
	\arrow[ll, bend left=30,"j_*"{description}] 
	\arrow[ll, bend right=30, "j_!"{description}]
\end{tikzcd},\end{equation}
where $j_!$ is the inclusion functor, $j^*=T_S^*$ and $j_*=T_S|_{\lsup{\perp}{\langle S\rangle}}$.
\end{lemma}

\subsection{Serre functor} Let $\C$ be a Hom-finite $\k$-linear category, where $\k$ is an algebraically closed field. A \emph{Serre functor} is a $\k$-linear autoequivalence $\S_{\C}$ of $\C$ such that for any objects $A,B\in \C$, there exists an isomorphism
\[\eta_{A,B}\colon \Hom(A,B)\to \Hom(B,\S_{\C} A)^*,\]
where $(-)^*=\Hom_\k(-,\k)$ is the standard duality.

Recall that an abelian category $\A$ is called \emph{Ext-finite} if $\dim_{\k}\Ext^i_\A(X,Y)<\infty$ for any $X,Y \in \A$ and all $i\in \mathbb Z$. An Ext-finite abelian category $\mathcal A$ has \emph{Serre duality} if $D^b(\mathcal A)$ admits a Serre functor $\S$. 
A t-structure $(\mathcal U,\mathcal V)$ or an aisle $\U$ is called \emph{closed under the Serre functor} $\S$ if $\S\U\subseteq\U$.  

For a finite-dimensional algebra $\Lambda$ with a finite global dimension, the category $\A=\mod\Lambda$ has Serre duality. The Serre functor $\S$ is given by $\S\cong -\otimes_{\k}^L\Lambda^*\colon \db{\A}\to\db{\A}$. The category $\db{\A}$ has a Serre functor $\S$ if and only if $\db{\A}$ has Auslander-Reiten triangles. For every indecomposable object $X\in\db{\A}$, the Auslander-Reiten triangle is 
$\S[-1]X\to Y\to X\to \S(X)$, see \cite[Proposition~I.2.3]{RV1} and \cite[Proposition~2.8]{C}. 

\begin{example}[{\cite[Example~10.2]{SR16}}]
    Let $\Lambda$ be a finite-dimensional hereditary algebra and $(\T,\F)$ be a torsion pair in $\mod\Lambda$. The aisle $\U_{\mathcal T}$ determined by $\T$ is closed under the Serre functor if and only if for each projective object $P\in\mathcal T$, the corresponding injective object $P\otimes_{\Lambda}\Lambda^*$ is in $\mathcal T$. 
\end{example}

The following result is a special case of  \cite[Proposition 3.7]{BK}.
\begin{lemma}\label{lem:serre}
	Let $\mathcal A$ be an Ext-finite abelian category of finite global dimension and $S\in\db{\A}$ an exceptional object. Let $\S$ be the Serre functor of $\db{\A}$. Then $\ps$ and $\langle S\rangle^{\perp}$ have Serre functors given by 
	\[\S_{\ps}\cong T^*_S\circ\S,\; \S^{-1}_{{\langle S\rangle}^{\perp}}\cong T_S\circ \S^{-1}.\]
\end{lemma}

\section{Effaceable torsion pairs}\label{sec:three}

\subsection{Effaceable torsion pairs }
Let $\A$ be an abelian category with a torsion pair $(\T,\F)$,  and $\B$ be the HRS-tilt of $\A$.  The realization functor $\Phi\colon \db{\B}\to \db{\A}$ is not an equivalence in general. By \cite[Theorem A]{CHZ}, the realization functor $\Phi$ is an equivalence if and only if the torsion pair $(\T,\F)$ is effaceable. In \cite[Section 3]{CHZ}, there are several different characterizations of $(\T,\F)$ being effaceable. Now we give some characterizations of effaceable torsion pairs.

\begin{proposition}\label{prop:ff2}Let $\A$ be an abelian category with a torsion pair $(\T,\F)$,  and $\B$ be the HRS-tilt of $\A$. Assume that $A\in \mathcal{A}$. Then the following statements are equivalent:
	\begin{enumerate}
		\item The torsion pair $(\T,\F)$ is effaceable.
		\item There is an exact triangle $A\rightarrow B_1\rightarrow B_2\rightarrow \Sigma(A)$ in $\db\A$ with $B_1, B_2\in \mathcal{B}$;
		\item For any object $T\in\T$ and $F\in \F$, then any morphism $F\to T[1]$ in $\db\A$ factors as $F\to B\to T[1]$  with $B\in \B$.
		\item For $T\in\T$ and $F\in \F$ and any morphism $F[1]\to T[2]$ in $\db\A$,  there is an epimorphism $C\to F[1]$ in $\B$ such that the composition $C\to F[1]\to T[2]$ is zero.
		\item For $T\in\T$ and $F\in \F$ and any morphism $F[1]\to T[2]$ in $\db\A$,  there is a monomorphism $T\to D$ in $\B$ such that the composition $F[1]\to T[2]\to D[2]$ is zero. 
	\end{enumerate}
\end{proposition}
\begin{proof}
	The equivalence between $(1)$ and $(2)$ is \cite[Proposition 3.2]{CHZ}.
	
``(2)$\Rightarrow $(3)"  Given $T\in \T$ and $F\in \F$, any morphism $f\colon F\to T[1]$ in $\db{\A}$ could be complete to a triangle
$ T\to A\to F\xrightarrow{f} T[1]$ with $A\in \A$.  By assumption, we have another triangle  $A\rightarrow B_1\rightarrow B_2\rightarrow \Sigma(A)$ in $\db\A$ with $B_1,B_2\in \mathcal{B}$. Applying the octahedral axiom, we have the following commutative diagram
\begin{equation}\label{com:oct}
	\begin{tikzcd}
		&B_2[-1] \arrow[r,equal] \arrow[d] & B_2[-1] \arrow[d] &\\
		T \arrow[r]  \arrow[d,equal]& A\arrow[r] \arrow[d]&F\arrow[r, "f"] \arrow[d]&T[1]\arrow[d,equal]\\
		T \arrow[r] & B_1\arrow[r] \arrow[d]&B \arrow[d]\arrow[r]&T[1]\\
		&B_2\arrow[r, equal]&B_2&
	\end{tikzcd}
\end{equation}
Taking the cohomology of the third row and third column, we get an epimorphism
$H^0(B_1)\to H^0(B)$  and a monomorphism $H^{-1}(B)\to H^{-1}(B_2)$ in $\A$. It implies that $B\in \B$. Thus $f\colon F\to T[1]$ factors as $F\to B\to T[1]$ with $B\in \B$.

``(3)$\Rightarrow $(2)"  	 Assume that for any $T\in \T$ and $F\in \F$, any morphism $ F\to T[1]$ in $\db{\A}$ factors as $F\to B\to T[1]$.  Given $A\in \A$, there exists a canonical triangle $T\to A\to F\xrightarrow{f} T[1]$ with $T\in \T$ and $F\in\F$. Then we have the right-most commutative square in the commutative diagram (\ref{com:oct}). By the octahedral axiom, we obtain the commutative diagram (\ref{com:oct}). By $T,B\in \B$, the triangle of the third row implies that $B_1\in \B$.  Consider the third column, we have $B_2\in \B$ since $F[1], B \in \B$. The second column gives a triangle $A\to B_1\to B_2\to A[1]$ with $B_1,B_2\in \B$.

``(3)$\Rightarrow $(4)"   Given $T\in\T$, $F\in \F$ and any morphism $F[1]\to T[2]$ in $\db\A$,  there exists $B\in \B$ such that $F[1]\to T[2]$  factors as $F[1]\to B[1]\to T[2]$. By the octahedral axiom, we have the following commutative diagram 
\begin{equation}\label{com:ft}
	\begin{tikzcd}
		F \arrow[r]  \arrow[d,equal]& B\arrow[r] \arrow[d]&C\arrow[r, "f"] \arrow[d,dashed] &F[1]\arrow[d,equal]\\
		F \arrow[r] & T[1]\arrow[r]&A[1] \arrow[r, "g"]&F[1]
	\end{tikzcd}
\end{equation}
Since $B,F[1]\in\B$, we have that $C\in B$ and the morphism $C\to F[1]$ is an epimorphism in $\B$. Moreover, the composition $C\to F[1]\to T[2]$ is zero.

``(4)$\Rightarrow $(3)"    For any object $T\in\T$ and $F\in \F$, then any morphism $F[1]\to T[2]$ in $\db\A$,  there is an triangle $T[1]\to A[1]\to F[1]\xrightarrow{g[1]} T[2]$ in $\db{\A}$ with $A\in \A$.  If there is an epimorphism $C\to F[1]$ in $\B$ such that the composition $C\to F[1]\to T[2]$ is zero, then the morphism $C\to F[1]$ factors through $ g\colon A[1]\to F[1]$.  The right-most commutative square in (\ref{com:ft}) induces a morphism of triangles in (\ref{com:ft}). The top triangle $B\to C\to F[1]\to B[1]$ implies that  $B\in \B$. The morphism $F[1]\to T[2]$ must factor through $B[1]$.

The proof of equivalence ``(3)$\Leftrightarrow $(5)"  is similar to the proof of equivalnece``(3)$\Leftrightarrow $(4)". 
\end{proof}
\begin{remark}
	The equivalence of $\Phi\colon \db{\B}\to \db\A$ is controlled by the property of  morphisms $F\to T[1]$. The conditions (4) and (5) in  Proposition \ref{prop:ff2} verify the effaceable property of morphisms $F\to T[1]$.
\end{remark}
                                                                              
\subsection{Closedness of Serre functor}
Let $\A$ be an abelian category with torsion pair $(\T,\F)$. The HRS-tilt is $\B$, the heart of the t-structure $(\U_\T,\V_{\T})$.  Our goal is to show that if $\A$ has Serre duality and $(\T,\F)$ is efffaceable, then $\S\U_\T\subseteq \U_{\T}$.

\begin{proposition}\label{prop:epi}
Let $\A$ be an Ext-finite abelian category with Serre duality and   $(\mathcal T,\mathcal F)$ be an effaceable torsion pair of  $\A$. Let $\B$ be HRS-tilt of  $\A$. Given any morphism $B\to X[1]$ with $B\in \B$ and $X\in \U_\T$, there is an epimorphism $C\to B$ in $\B$ such that the composition $C\to B\to X[1]$ is zero.
\end{proposition}
\begin{proof}
	If  $(\T,\F)$ is effaceable and any $A\in \A$,  then there is a triangle $A\to B_1\xrightarrow{f} B_2\to A[1]$ with $B_1, B_2\in B$ . Complete $f$ to a triangle in $\db{\B}$, we have  $Z\to B_1\xrightarrow{f} B_2\to Z[1]$. Aplying  the realization functor  $\Phi\colon \db{\B}\to \db{\A}$ to the triangle, we have that $A\cong \Phi(Z)$. Thus $\A$ is contained in the essential image $\im G$. It follows that realization functor $\Phi\colon \db{\B}\to\db{\A}$ is an equivalence \cite[Proposition 3.3]{CHZ}. 
	
	We claim that for any morphism $\alpha\colon B\to D[n]$ with $B,D\in \B$ and $n\geq 1$, there exists an epimorphims $C\to B$ in $\B$ such that $C\to B\to D[n]$ is zero.  We complete the morphism $\alpha$ to a triangle in $\db{\A}$
	\[B'\xrightarrow{\alpha'}B\xrightarrow{\alpha} D[n]\to B'[1].\]
	Assume that $B'\cong \Phi(E)$ with $E\in \db{\B}$. Applying cohomological fucntor $H^i_{\B}$, we have that $H^i_{\B}(\Phi(E))= 0$ for $i\neq 0,-n+1$. By \cite[Lemma~2.3.(2)]{CHZ}, we have that the standard cohomology in $\db{\B}$  satisfy $H^i(E)=0$ for $i\neq 0,-n+1$. We assume that $E$ is of the form in $\db{\B}$
	\[\cdots\to 0\to E^{-n+1}\to \cdots\to E^0\to 0\to \cdots\]
	Let $\beta\colon E^0\to E$ be the canonical morphism. The composition $\beta'\colon E^0\xrightarrow{\Phi(\beta)} \Phi(E)\to B$  is an epimorphism in $\B$. Indeed, $H_{\B}^0(E^0)\to H^0_{\B}(\Phi (E))$ is epimorphism and $H^0_{\B}(\Phi(E))\to H^0_{\B}(B)$ is an isomorphism. Let $C=E^0$, and the epimorphism $\beta'\colon C\to B$ satisfies $C\to B\to D[n]$ is zero.
	
	Let $X\in\U_{\T}$. Since $\U_\T$ is bounded there exists non-negative integers $m\leq n$, with $n$ minimal and $m$ maximal, such that $H^i_\B(X)=0,i\notin[-n,-m]$. We use induction on $n-m$. If $n-m=0$ then $X\in\B[n]$ and the assertion follows by the above discussions. Suppose $n-m>1$. The choice of $m$ shows that $H^{-m}_\B(X)\neq 0$. So there is a triangle
	\[\tau^\B_{\leq -m-1}X[1]\to X[1]\to H^{-m}_\B(X)[m+1]\to \tau^\B_{\leq -m-1}X[2]\] 
	For any morphism $B\to X[1]$, there exists an epimorphism $C'\to B$ in $\B$ such that $C'\to B\to X[1]\to H^m_\B(X)[m+1]$ is zero. Thus the composition morphism $C'\to X[1]$ factor as $C'\to \tau^\B_{\leq -m-1}X[1]$. By induction argument, there exists an epimorphism $C\to B$ in $\B$ such that $C\to B\to X[1]$ is zero.
\end{proof}

Now we can prove the aisle $\U_{\T}$ is closed under Serre functor, compare to \cite[Proposition 4.12]{SR16}.
\begin{theorem}\label{thm:suffr}
	Let $\A$ be an Ext-finite abelian category with Serre duality. If a torsion pair $(\mathcal T,\mathcal F)$ on $\A$ is effaceable. Then the aisle $\U_\T$ satisfies that $\S\U_\T\subseteq \U_{\T}$.
\end{theorem}
\begin{proof}
The aisle $\U_\T$ is closed under positive shifts and extensions. If there exists $X\in \U_\T$ such that $\S X\notin \U_\T$, then there exists $C\in \B$ such that $\S C\notin \U_\T$.  

Let $m$ be the maximal number such that $H^m_{\B}(\S C)\neq 0$ with $m\geq 1$. Then there is a triangle 
\[\tau^\B_{\leq m-1}(\S C)\to \S C\xrightarrow{h} H^m_{\B}(\S C)[-m]\xrightarrow{g[-m]} \tau^\B_{\leq m-1}(\S C)[1]\]
Since $\tau^\B_{\leq m-1}(\S C)[m+1]\in \U_\T[1]$,  by Proposition \ref{prop:epi}, there exist an epimorphism  $B\to H^m_{\B}(\S C)$ in $\B$ such that $B\to H^m_{\B}(\S C)\xrightarrow{g} \tau^\B_{\leq m-1}(\S C)[m+1]$ is zero. It follows that $B[-m]\to H^m_{\B}(\S C)[-m]$ factors through $h\colon \S C\to H^m_{\B}(\S C)[-m]$. Thus $\Hom(B[-m],\S C)\neq 0$. By Serre duality, we have $\Hom(C,B[-m])\neq 0$. This contradics with $\Hom(C,B[-m])=0$ for any $B,C\in \B$ and $m\geq 1$.
\end{proof}
The above result implies the following about the standard t-structure.
\begin{corollary}
	Let $\A$ be an Ext-finite abelian category with Serre duality. Then the standard t-structure $(\D^{\leq 0},\D^{\geq 0})$ of $\db{\A}$ satisfies that $\S \D^{\leq 0}\subseteq \D^{\leq 0}$.
\end{corollary}

\section{Sufficient for torsion pair being effaceable}\label{sec:sp}
    
For a finite-dimensional algebra $\Lambda$, we denote the category $\mod \Lambda$ by $\A$. We have proved that for an algebra $\Lambda$ of finite global dimension, the torsion pair $(\T,\F)$ of $\A$ is effaceable implies that the aisle $\U_{\T}$ in $\db{\Lambda}$ is closed under the Serre functor. In the following sections, we prove that the converse holds in some special cases. We consider the torsion classes that are finitely generated or without Ext-projective objects. The statement for general cases of torsion classes depends on the factorization of $F[1]\to T[2]$ with $F\in F$ and $T\in \T$. 
\subsection{ Aisle determined by special torsion classes}	
Recall that an object $E\in\db{\A}$ is called \emph{presilting} if $\Hom_{\db{\A}}(E,E[1])=0${, and \emph{silting} if further $\thick E=\db{A}$.} A bounded t-structure $(\U,\V)$ is called \emph{finitely generated} if there exists an object $E$ such that $\U=\{X\in\db{A}|\Hom(E,X[>0])=0\}$. The object $E$ is called a \emph{generator} of $\U$.

Let $\Lambda$ be an algebra of finite global dimension. By \cite[Lemma 4.6]{LVY}, the aisle $\U_T$ associated to a silting object $T$ in $\db{\Lambda}$ is closed under Serre functor if and only if the object $T$ is a tilting object. Moreover, if $\U_T$ corresponds a torsion class $\T$ in $\mod\Lambda$, then the torsion pair $(\T,\F)$ must be effaceable by \cite[Theorem A]{CHZ}. In the following result, we explicitly construct the 5-term exact sequence for each $\Lambda$-module $X$ to show the torsion pair is effaceable. Note that the 5-term exact sequence is not unique for the $\Lambda$-module $X$.
\begin{lemma}\label{lem:enough}
	Let $\Lambda$ be a finite-dimensional algebra of finite global dimension. Assume $\T$ is a finitely generated torsion class in $\mod \Lambda$ and the aisle $\U_\T$ is closed under the Serre functor $\S$. Then the torsion pair $(\T,\F)$ is effaceable.
\end{lemma}
\begin{proof}
	Recall that for each t-structure $(\U,\V)$ with $\U$ finitely generated, we have that $(\V,\S\U)$ is a co-t-structure, see \cite[Proposition~2.22]{AI},\cite[Theorem 4.3.2]{Bo}. 
	
	If $\mathcal T$ is finitely generated, then so is $\U_{\T}$. Now we take the the co-t-structure $(\V_{\T},\S\U_{\T})$. For any $X\in\A$, let 
	\[N\to X\to M\to N[1]\]
	be the triangle of $X$ with respect to $(\V_\T,\S\U_\T)$. Taking cohomology with respect to the canonical t-structure, we have the following long exact sequence
	\[0\to H^{-1}(M)\to H^0(N)\to X\to H^0(M)\to H^1(N)\to 0\]
	where the first term $H^{-1}(M)$ lies in $\F$ because $N\in\V_{\T}$ and $H^0(N)\in\F$. Similarly, $H^0(M)\in\T$ implies that $H^1(N)\in\T$. Thus $(\T,\F)$ is effaceable.
\end{proof}

If a torsion class $\T$ is not finitely generated, then we do not have the co-t-structure $(\V_{\T},\S\U_{\T})$. In this case, there is no direct approach to prove $(\T,\F)$ is effaceable under the assumption $\U_{\T}$ is closed under the Serre functor. For the category $\A$ with $\Lambda$ being a hereditary algebra, we could show that $\S\U_{\T}\subseteq \U_{\T}$ implies $(\T,\F)$ is effaceable under the assumption the torsion class $\T$ without Ext-projectives.

\begin{lemma}\label{lem:torp}
	Assume that $\mathcal A$ is a hereditary category and $\T$ is a torsion class without Ext-projective objects, then $(\T,\F)$ is effaceable.
\end{lemma}
\begin{proof}
	For any projective object $P\in \A$, there is a short exact sequence
	\[0\to P_{\T}\to P\to P_{\F}\to 0\]
	with $P_{\T}\in \T$, and $P_{\F}\in \F$. Since $\A$ is hereditary, we have that $P_{\T}$ is either zero or projective. Since $\T$ does not contain any nonzero projective objects, we have that $P_{\T}=0$ and hence $P\in\F$.	 Then each object $A\in \A$ has a projective resolution $0\to P_1\to P_2\to A\to 0$ with $P_1,P_2\in \F$.
\end{proof}

\subsection{The factorization}
In this section, we assume that category $\mathcal A$ is an abelian category with Serre duality $\S$. Let $(\T,\F)$ be a torsion pair of $\mathcal A$ such that the corresponding aisle $\U_{\T}$ is closed under the Serre functor, i.e. $\S\U_{\T}\subset \U_{\T}$. The HRS-tilt of $\A$ with respect to $(\T,\F)$ is denoted by $\B$. 
Let $E\in \U_{\T}$ be an indecomposable Ext-projective object.
The following result shows that the indecomposable Ext-projective object $E$ is projective in $\B$ and has a simple top $S$. 

\begin{proposition}[{\cite[Proposition~6.4]{SR16}}]\label{prop:SE}
	Let $\A$ be an abelian category with Serre duality and $(\U,\V)$ the bounded t-structure in $\db\A$ closed under the Serre functor. Let  $E$ be an Ext-projective of $\U$, then  we have that
	\begin{enumerate}
		\item $E$ is projective in $\B$ and $\S E$ is injective in $\B$;
		\item  The object $E$ has a simple top $S$ in $\B$;
		\item If $\A$ is hereditary, then                         $\Hom_{\db{\A}}(S,S[n])=0$ for $n\neq 0$, and $\End S=\k$.
	\end{enumerate}    
\end{proposition}

We will consider an indecomposable Ext-projective object $E$ in $\U_{\T}$ and its simple top $S$ in $\B$. For objects $A,B\in \B$, we consider when a morphism $A\to B[2]$ in $\db{\A}$ could be factored as $A\to B'[1]\to B[2]$ with $B'\in \B$.
The following results are essentially in \cite{SR16}. For the reader's convenience, we give proofs and explicit constructions here.
\begin{lemma} \label{lem:zero}
	Given $B\in \B$ and $E$ a indecomposable projective object in $\B$ with simple top $S$, Let $K_B$ be the kernel of $B\to \Hom(B,\S E)^*\otimes\S E$ in $\B$, Then $\Hom(S,K_B)=0$.
\end{lemma}
\begin{proof}
   we claim that $\Hom(K_B,\S E)=0$. Indeed, each morphism $K_B\to \S E$ could be lifted to  $B\to \S E$, thus the composition $K_B\to B\to  \S E$ is zero.
   
	If $S\to K$ is non-zero, then the morphism is a monomorphism. By the injectivity of $\S E$, there is a non-zero morphism $K_B\to \S E$. It is impossible.
\end{proof}

\begin{lemma}\label{lem:conM}
	Given any morphism $A\to B[2]$ with $A,B\in \B$, there exists an epimorphism $M\to A$ in $\B$ and a morphism $M\to K_B[2]$ satisfy the following property:
	  If the morphism $M\to K_B[2]$ has a factorization  $ M\to B''[1]\to K_B[2] $ in $\B$ then the morphism $A\to B[2]$ has a factorzation $A\to B'[1]\to B[2]$ in $\B$.
\end{lemma}
\begin{proof}
Let $D$ be the mapping cone of the evaluation morphism  $B\to\Hom(B,\S E)^*\otimes \S E$. Let $K_B$ and $C_B$ be the kernel and cokernel of $B\to\Hom(B,\S E)^*\otimes \S E$ in $\B$ respectively. We have two triangles in $\db{\A}$
\[B\to\Hom(B,\S E)^*\otimes \S E\to D\to B[1],\quad K_B[1]\to D\to C_B\to K_B[2] \]
Since the map $A\xrightarrow{f} B[2]\to \Hom(B,\S E)^*\otimes \S E[2]$ is zero by Serre duality, we have that $A\to B[2]$ factor as $A\xrightarrow{f'} D[1]\to B[2]$. By Octahedron axiom, the composition $A\to D[1]\to C_B[1]$  gives the following morphism of triangles

\begin{equation}\label{com:conM}
	\begin{tikzcd}
		C_B \arrow[r]  \arrow[d,equal]& M\arrow[r] \arrow[d]&A\arrow[r] \arrow[d,"f'"]&C_B[1]\arrow[d,equal]\\
		C_B \arrow[r] & K_B[2]\arrow[r]&D[1] \arrow[r]&C_B[1]
	\end{tikzcd}
\end{equation}
Since $C_B, A\in \B$, the top triangle implies that $M\in \B$. 
If there exists an epimorphism $C\to M$ in $\B$ such that the composition $C\to M\to K_B$ is zero, then the middle square implies that $C\to M\to A$ is an epimorphism such that $C\to M\to A\to B[2]$ is zero.
\end{proof}

Let $K_M$ and $C_M$ be the kernel and cokernel of the morphism  $\Hom(E,M)\otimes E\to M$ in $\B$, respectively. The mapping cone of $\Hom(E,M)\otimes E\to M$ is denoted by $\tilde{D}$. Note that $H^0_{\B}(\tilde D)\cong C_M$ and $H^{-1}_{\B}(\tilde D)\cong K_M$. We have two triangles in $\db{\A}$
\[\tilde{D}[-1]\to \Hom(E,M)\otimes E\to M\to \tilde{D},\quad K_M[1]\to \tilde{D}\to C_M\to K_M[2]\]

Since $E$ is projective in $\B$, we have that $\Hom(E,K_B[2])=0$. It follows that the map $M\to K_B[2]$ factors as $M\to \tilde{D}\to K_B[2]$. The composition $K_M[1]\to \tilde{D}\to K_B[2]$ gives the following morphism of triangles
\begin{equation}\label{com:km}
	\begin{tikzcd}
	K_M[1] \arrow[r]  \arrow[d,equal]&\tilde{D}\arrow[r] \arrow[d]&C_M\arrow[r] \arrow[d,"f'"]&K_M[2]\arrow[d,equal]\\
		K_M[1] \arrow[r] & K_B[2]\arrow[r]&N[2] \arrow[r]&K_M[2]
	\end{tikzcd}
\end{equation}
The bottom triangle shows that $N\in \B$ and there exists a monomorphism $K_B\to N$ in $\B$.
 
\begin{lemma}\label{lem:hom}Keep the notations as above. Assume that $\A$ is hereditary then we have $\Hom(C_M,S)=0$ and $\Hom(S,N)=0$.
\end{lemma}
\begin{proof}
	Note that $\Hom(E,C_M)=0$.  Indeed each morphism $E\to C_M$ factors as $E\to M\to C_M$ since $E$ is projective in $\B$. By the construction of $C_M$, we have that $E\to C_M$ is zero. If there exists a non-zero morphism $C_M\to S$, then the canonical morphism $E\to S$ could be lifted to a non-zero morphism $E\to C_M$. Thus $\Hom(C_M,S)=0$.
	
	We show that $\Hom(S,N)=0$. By Lemma \ref{lem:zero}, $\Hom(S,K_B)=0$, it sufficient to show that $\Hom(S,K_M)=0$.  We could assume that $\Hom(E,M)\neq 0$.  If $\Hom(S,K_M)\neq 0$, then there is a nonzero composition $ S\to K_M\to \Hom(E,M)\otimes E $ and hence a non-zero morphism $E\to S\to E$. Since $E$ is indecomposable projective, it follows that $S\cong E$. This implies that the composition $S\to K_M\to  \Hom(E,M)\otimes E$ is a split monomorphism such that the composition $S\to  \Hom(E,M)\otimes E \to M$ is zero, which contradicts the universal property of the evaluation map $\Hom(E,M)\otimes E \to M$.
  	\end{proof}

\begin{proposition}\label{prop:fac}Keep the notations as above. Assume that $\A$ is hereditary. 
If the morphism $f'\colon C_M\to N[2]$ factors as $C_M\to B'[1]\to N[2]$ with $B'\in \B$, then the map $M\to K_B[2]$ factors as $M\to B''[1]\to K_B[2]$  with $B''\in \B$.
\end{proposition}
\begin{proof}
	If the morphism $f'\colon C_M\to N[2]$ factors as $C_M\to B'[1]\to N[2]$, then there is a triangle $B'[1]\to N[2]\xrightarrow{g[2]} C'[2]\to B'[2]$, with $C'\in \B$. The morphism $g\colon N\to C'$ is an monomorphism in $\B$. Moreover, the composition $C_M\to N[2]\to C'[2]$ is zero.
	
	By the construction of $\tilde{D}$ in commuative diagram (\ref{com:km}), we know the morphism $M\to K_B[2]$ factors through $ \tilde{D}\to K_B[2]$. Moreover, the composition $\tilde{D}\to K_B[2]\to N[2]\to C'[2]$ is zero. We take the cokernel $B''$of the composition of  monomorphisms $K_B\to N\xrightarrow{g} C'$ in $\B$. It follows that the morphism $\tilde{D}\to K_B[2]$ factor through $B''[1]\to K_B[2]$.
\end{proof}

\section{The Reduction via exceptional objects}\label{sec:five}
Throughout this section, we assume that category $\mathcal A$ is a hereditary abelian category with Serre duality $\S$. Let $(\T,\F)$ be a torsion pair of $\mathcal A$ such that the corresponding aisle $\U_{\T}$ is closed under the Serre functor, i.e. $\S\U_{\T}\subset \U_{\T}$. The HRS-tilt of $\A$ with respect to $(\T,\F)$ is denoted by $\B$.

In Section \ref{sec:sp}, we prove that either $\T$ is finitely generated or does not have Ext-projective objects, then $(\T,\F)$ is effaceable. In this section, we prove this fact for any torsion pair $\T$ such that $\S\U_{\T}\subset \U_{\T}$, $(\T,\F)$ is effaceable. We will prove the statement for any torsion pair $(\T,\F)$ such that $\S\U_{\T}\subset \U_{\T}$. 
We will show how the torsion pair and HRS-tilt interwine with the recollement given by some exceptional objects in $\db{\A}$.

\subsection{The aisle on \texorpdfstring{$\ps$}{}} Let $\A$ be an abelian category with Serre duality.
We consider an indecomposable Ext-projective object $E$ in an aisle $\U_{\T}$ of $\db\A$. 
Given a triangulated category $\D$ and a full additive category $\C$, an object $E$ in $\C$ is called Ext-projective if $\Hom_{\D}(E,C[i])=0$ for all $i>0$. 
By \cite[Theorem 1.4.10]{BBD}, t-structures could be glued with respect to a recollement. Given a recollement (\ref{eq:rec}) and t-structures $(\D_1^{\leq 0},\D_1^{\geq 0})$ and $(\D_2^{\leq 0},\D_2^{\geq 0})$ of $\C_1$ and $\C_2$, then there exists a t-structure $(\tilde{\D}^{\leq 0},\tilde{\D}^{\geq 0})$ on $\C$ such that functors $i_*, j^*$ are t-exact. The t-structure $(\tilde{\D}^{\leq 0},\tilde{\D}^{\geq 0})$  is defined as
\[\tilde{\D}^{\leq 0}=\{X\in \C\mid j^*(X)\in \D_2^{\leq 0},\; i^*(X)\in \D_1^{\leq 0}\};\]
\[\tilde{\D}^{\geq 1}=\{X\in \C\mid j^*(X)\in \D_2^{\geq 1},\; i^!(X)\in \D_1^{\geq 1}\}.\]
We call the t-structure $(\tilde{\D}^{\leq 0},\tilde{\D}^{\geq 0})$ on $\C$ the \emph{glued t-structure} from the t-structures on $\C_1$ and $\C_2$. 
\begin{proposition}\label{prop:glue-t}\cite[Proposition 1.4.12]{BBD}
	Let $(\tilde{\D}^{\leq 0},\tilde{\D}^{\geq 0})$ be a t-structure on $\C$, the following conditions are equivalent:
	\begin{enumerate}
		\item $j_!j^*(\tilde{\D}^{\leq 0})\subseteq \tilde{\D}^{\leq 0}$.
		\item $j_*j^*(\tilde{\D}^{\geq 0})\subseteq \tilde{\D}^{\geq 0}$.
		\item the t-structure of $\C$ is a glued t-structure from $\C_1$ and $\C_2$.
	\end{enumerate}
\end{proposition}
In this case, the t-structures on $\C_1$ and $\C_2$ are given by $(i^*\tilde{\D}^{\leq 0}, i^!\tilde{\D}^{\geq 0})$ and $(j^*\tilde{\D}^{\leq 0}, j^*\tilde{\D}^{\geq 0})$ respectively. 
Let $E\in \U_T$ be an indecomposable Ext-projective object which is projective in $\B$ and has a simple top $S$. Since $S$ is an exceptional object, it follows that there is a recollement (\ref{eq:recl}) by Lemma \ref{eq:recl}.
The following result is implicitly contained in \cite[Lemma 9.1, Proposition 9.2]{SR16}. We give a direct proof here for the reader's convenience.
\begin{lemma}\label{lem:glue-t}
	For $S$ being the simple top as above,  t-structure $(\U_\T,\V_\T)$ of $\db{\A}$ is glued  with respect to the recollement:
	\begin{equation}\label{eq:recf}
\begin{tikzcd}
		\langle S\rangle
		\arrow[rr, "i_*"{description}] 
		&& \db{\A}
		\arrow[ll,  bend left=30, "i^!_S"{description}] 
		\arrow[ll, bend right=30, "i^*_{S}"{description}]
		\arrow[rr, "j^*"{description}]
		& &\lsup{\perp}{\langle S\rangle}
		\arrow[ll, bend left=30, "j_*"{description}] 
		\arrow[ll, bend right=30, "j_!"{description}]
 	\end{tikzcd}
	\end{equation}  
	where $j^!$ is the inclusion functor, $j^*=T^*_S$ and $j_*=T_S|_{\lsup{\perp}{\langle S\rangle}}$. In this case, $T_S^*(\U_\T)=\U_\T\cap \lsup{\perp}{\langle S\rangle}$ is an aisle of $\lsup{\perp}{\langle S\rangle}$.
\end{lemma} 
\begin{proof}
	We know that $j^*=T^*_S$ and the following triangle for $X\in \U_{\T}$
	\[T^*_S(X)\to X\xrightarrow{\alpha_X} \rhom(X,S)^*\lten_\k S \to T^*_S(X)[1].\]
	By the fact that $\rhom(X,S)^*\lten_\k S=\oplus_{i\in\mathbb Z} \Hom(X,S[i])^*\otimes_\k S[i]$ and $S\in \U_\T\cap \V_\T$, we have $\Hom(X,S[i])=0$ for $i<0$. Thus 
	\[\rhom(X,S)^*\lten_\k S=\oplus_{i\geq 0} \Hom(X,S[i])^*\otimes _\k S[i]\in \U_\T.\]
	We still need to show that $T^*_S(X)\in \U_\T$. This equivalent to $H_\B^i(T^*_S(X))=0$ for $i>0$. Applying cohomological functor $H_\B$ to the above triangle, we have that $H^i_\B(T^*_S(X))=0$ for $i>1$. On the other hand, we have the following exact sequence 
    \[H^0_\B(X)\to H^0_\B( \rhom(X,S)^*\lten_\k S)\to H^1_\B(T^*_S(X))\to 0\]
    Since $S$ is simple in $\B$ and $H^0_\B(X)\to H^0_\B( \rhom(X,S)^*\lten_\k S)$ is epimorphism, we have that $H^1_\B(T^*_S(X))=0$. It follows that $T^*_S(\U_\T)\subseteq \U_\T\cap \ps$.
	
    Since $T^*_S|_{\ps}$ is identity, thus the aisle $\U_\T$ of $\db{\A}$ has image $T_S^*(\U_\T)=\U_\T\cap \lsup{\perp}{\langle S\rangle}$.  Thus, we have that 
	\[j_!j^*(\U_\T)=\U_\T\cap \lsup{\perp}{\langle S\rangle} \subseteq \U_\T.\]
	By Proposition \ref{prop:glue-t},  t-structure $(\U_\T,\V_\T)$ of $\db{\A}$ is glued from with respect to the recollement (\ref{eq:recf}).  Moreover, $T_S^*(\U_\T)=\U_\T\cap \lsup{\perp}{\langle S\rangle}$ is an aisle of $\ps$.
\end{proof}

\begin{remark}
	Note that $(j^*(\U_\T),j^*(\V_\T))$ is a t-structure of  $\lsup{\perp}{\langle S\rangle}$,  and  $j^*(\U_\T)=\U_\T\cap \lsup{\perp}{\langle S\rangle}$, but $j^*(\V_\T)$ does  not equal $\V_\T\cap \lsup{\perp}{\langle S\rangle}$ in general.
\end{remark}


By Lemma \ref{lem:serre}, the perpendicular category $\lsup{\perp}{\langle S\rangle}$ has a Serre functor $T^*_S\circ \S$. We prove the following statement about the Serre functor.
\begin{proposition}
	Let $S$ and $(\U_\T,\V_\T)$ as above, then $j^*(\U_\T)$ is closed under Serre functor $T_S^*\circ \S$ of $\lsup{\perp}{\langle S\rangle}$ if $\U_\T$ is closed under Serre functor $\S$.
\end{proposition}
\begin{proof}
Assume that $\S\U_\T\subseteq \U_{\T}$, we check $T^*_S\circ\S (\U_\T\cap \lsup{\perp}{\langle S\rangle}))\subseteq \U_{\T}\cap \lsup{\perp}{\langle S\rangle}$.
Indeed, $$T^*_S\circ\S (\U_\T\cap \lsup{\perp}{\langle S\rangle})\subseteq T^*_S(\U_\T)\cap T^*_S( \lsup{\perp}{\langle S\rangle})\subseteq \U_\T\cap \lsup{\perp}{\langle S\rangle}.$$

\end{proof}

\subsection{HRS-tilts and recollement}
Now we assume that the category $\A$ is a hereditary abelian category, and it follows that each object $X\cong\oplus_{i\in\mathbb Z}H^i(X)[-i]$. There is a torsion pair  $(\T,\F)$ of $\A$ and HRS-tilt of $\A$ is $\B$.

 For the simple object $S\in \B$,  either $S\in \A$ or $S[-1]\in \A$. By Proposition \ref{prop:SE}, the object $S\in \db{\A}$ is exceptional, thus the category $\ps$ is an admissible subcategory of $\db{\A}$. We define a subcategory $\W$ of $\A$ is the intersection $\ps\cap \A$. There exists an object $R$ of  $\A$ such that 
\[\W=\{X\in\A|\Hom(X,R)=0,\Ext^1_\A(X,R)=0\},\]
where $R=S$ or $R=S[-1]$.
By \cite[Proposition 1.1]{GL}, $\W$ is an exact subcategory of $\A$. Moreover, the category $\W$ is equivalent to a module category of a finite-dimensional hereditary algebra with one fewer distinct simple object of  $\A$  \cite[Proposition 3]{HRS2}.

\begin{lemma}\label{lem:res}
	The restriction of standard t-structure $(\D^{\leq 0},\D^{\geq 0})$  of $\db{\A}$ on $\ps$ is a bounded t-srtucture with heart $\W$.
\end{lemma}
\begin{proof}
We only need to check each object $X\in \ps$, there is a finite integers $ k_1 > k_2 > \cdots > k_n$ such that  $0=X_n\to X_{n-1}\to \cdots\to X_0=X$ with each mapping cone of $X_i\to X_{i-1}$ lying in $\W[k_i]$ for $1\leq i\leq n$. For any object $X\in \ps$, $X\cong\oplus_{i=1}^{n}H^{k_i}(X)[-k_i]$. It follows that $\Hom(H^{k_i}(X),S[m])=0$ for $1\leq i\leq n
$ and $m\in \mathbb Z$. This implies that $H^{k_i}(X)\in \W$ for any $1\leq i\leq n$.
\end{proof}

\begin{proposition}\label{prop:real}
The embedding $\W\to \ps$ could be lifted to an equivalent t-exact functor $\db{\W}\to \ps$ such that the restriction on $\W$ is identity. 
\end{proposition}
\begin{proof}The category $\ps$ is an algebraic triangulated category \cite[Lemma 7.5]{Kr1}. Thus, by Lemma~\ref{lem:res}, there exists a realization functor $\Phi\colon \db{\W}\to \ps$ \cite[3.2]{KV}. We need to prove this functor $\Phi$ is an equivalence. 

For any object $X,Y\in \W$, we have that $\Hom_{\db{\W}}(X,Y[i])=0=\Hom_{\db{\A}}(X,Y[i])$ for $i\neq 0,1$. Thus, the realization functor $\Phi$ is fully faithful by \cite[Remark~3.1.17]{BBD}. Since $\W$ is bounded heart in $\ps$, $\Phi$ is dense. So the assertion follows. 
\end{proof}

Combining the recollement (\ref{eq:recf}) and Proposition \ref{prop:real}, we have a recollement (\ref{eq:recm})  which is obtained from the recollement (\ref{eq:recf}) by composing the quasi-inverse of the equivalence $\db{\W}\to \ps$.
	\begin{equation}\label{eq:recm}
	\begin{tikzcd}
		\db{\k}
		\arrow[rr, "i_*"{description}] 
		&& \db{\A}
		\arrow[ll, bend left=30, "i^!_S"{description}] 
		\arrow[ll, bend right=30, "i^*_{S}"{description}]
		\arrow[rr, "j^*"{description}]
		& &\db{\W}
		\arrow[ll, bend left=30,  "j_*"{description}] 
		\arrow[ll, bend right=30, "j_!"{description}]
	\end{tikzcd}
\end{equation} 
\begin{remark}
	In general, $j^*(\A)\ncong \W$. The reason is that the standard t-structure of $\db{\A}$ is not the glued t-structure of the standard t-structure of $\db{\W}$.
\end{remark}

\begin{lemma}\label{lem:tor}keep the notation as above,
given a torsion pair $(\T,\F)$ of $\A$, then $\T\cap\W$ is a torsion class of $\W$.  
\end{lemma}
\begin{proof}
	The aisle $\U_{\T}$ of $\db{\A}$ determined by $\T$ satisfies that $\D^{\leq 0}[1]\subseteq \U_{\T}\subseteq \D^{\leq 0}$. By Lemma \ref{lem:res}, we have that $\D^{\leq 0}\cap \ps$ is an aisle of  $\ps$. By Lemma \ref{lem:glue-t}, $\U_\T\cap \ps$ is an aisle of $\ps$. Thus, we have the following inclusions of aisles of $\ps$
	\[(\D^{\leq 0}\cap \ps)[1]\subseteq \U_{\T}\cap \ps \subseteq \D^{\leq 0}\cap \ps.\] On the other hand, the heart of the t-structure $(\D^{\leq 0}\cap \ps, \D^{\geq 0}\cap \ps)$ is $\W$, see Lemma \ref{lem:res}. Thus $\U_{\T}\cap \ps\cap \W=\T\cap \W$ is a torsion class of $\W$ by Proposition \ref{prop:int}. 
\end{proof}

The subcategory $\T':=\T\cap \W$ is a torsion class in $\W$. Write the corresponding torsion-free class as $\F'$. The category $\B'=\F'[1]\ast \T'$ is the HRS-tilt of $\W$.
\begin{remark}
	The torsion class $\T'$ in $\W$ is given by $\T\cap\W$ but the corresponding torsion free class $\F'\neq \F\cap \W$ in general.
\end{remark}
\begin{lemma}\label{lem:tred}
	Keep the notation as above, we have that $T^*_S(\B)\cong\B'$. Moreover, $T^*_S(\T)\cong \T'$.
\end{lemma}
\begin{proof}By Lemma \ref{lem:glue-t}, the t-structure $(\U_\T,\V_{\T})$ is glued from t-structure $(\U_{\T}\cap\ps, j^*(\V_{\T}))$ on $\ps$ and $\langle S\rangle$. Thus the heart of the t-structure $j^*(\U_\T)\cap j^*(\V_{\T})$ is $j^*(\B)=T^*_S(\B)$. By Lemma \ref{lem:res} and Lemma \ref{lem:tor}, the corresponding heart is also $\B'$. Thus we have 
	$T^*(\B)\cong \B'$.
	Note that $\T=\B\cap \U_\T$, we have that $T_S^*(\B\cap \U_\T)=\B'\cap \U'_\T=\T'$.
\end{proof}

\begin{remark}
In the recollement (\ref{eq:recm}), the t-structure of $\db\A$ by the HRS-tilt is glued along the functor $j^*$ from the t-structure of $\db\W$ determined by the corresponding HRS-tilt. While the standard t-structure of $\db{\A}$ is not glued from this recollement. By \cite[Theorem 6.3]{LV}, there exists another different recollement such that the standard t-structure is a glued t-structure. Such a recollement is obtained via reflecting of (3.1), see \cite{LVY, Jor}. In general, it is impossible to apply \cite[Theorem 6.4]{LVY} in the above setting.
\end{remark}

\subsection{Aisle determined by general torsion classes}
In this section, we set $\A$ the module category of hereditary algebra $\Lambda$ and $(\T,\F)$ a torsion pair of $\A$.  The HRS-tilt is $\B$.  For an $E\in \T$ is an Ext-projective object which is also Ext-projective in $\U_\T$. Let $S\in \B$ be the corresponding simple object.  For $\W=\A\cap \ps$, we have that $\ps$ is equivalent to $\db{\W}$ by Proposition \ref{prop:real}. Lemma \ref{lem:tor} shows that $\W$ admits a torsion pair $(\T',\F')$ with $\T'=\T\cap \W$. Let $\B'=\F'[1]\ast\T'$ the HRS-tilt of $\W$.
In this section, we will prove Proposition~\ref{prop:main} and Theorem \ref{thm:main}.  

Recall that for any $A,B\in \B$ the factorization of a morphism $A\to B[2]$ in $\db{\A}$ could be reduced to the factorization the morphism $T^*_S(A)\to T^*_S(B)[2]$ in $\ps$. This fact depends on the following result. 

\begin{proposition}\label{prop:red}\cite[Proposition 9.6]{SR16}
    Keep the notations as above. Assume that$\S\U_{\T}\subseteq\U_{\T}$. Given a morphism $A\xrightarrow{f}B[2]$ in $\db\A$with $A,B\in \B$ and $\Hom(A,S)=0=\Hom(S,B)$.
	If there exists an epimorphism $C'\to T^*_S(A)$ in $\B'$ such that the composition  $C'\to T^*_S(A)\xrightarrow{T^*_S(f)} T^*_S(B)[2]$, then there is an epimorphism $H^0_{\B}(C')\to A$ such that $H^0_{\B}(C')\to A\xrightarrow{f} B[2]$ is zero. 
\end{proposition} 

\begin{proposition}\label{prop:main}
	If the torsion pair $(\T',\F')$ of $\W'$ is effaceable, then $(\T,\F)$ is effaceable in $\A$.
\end{proposition}
\begin{proof}
	By Proposition \ref{prop:ff2}, it is equivalent to prove any map $f\colon F\to T[1]$ with $F\in \F, T\in\T$ factors through some morphism $C\to T[1]$ with $C\in \B$.
	Note that $F[1]\in \B$ and $T\in \B$, we claim that the morphism $f[1]\colon F[1]\to T[2]$ factors through some $C[1]\to T[2]$ for $C\in \B$. 
	
	By Lemma \ref{lem:hom} and Proposition \ref{prop:fac}, there exists $C_M$ and $N$ in $\B$ satisfying that $\Hom(C_M,S)=0=\Hom(S,N)$ such that $C_M\to N[2]$ factors through $C'[1]\to N[2]$ with $C'\in \B$ implies that $M\to K_T[2]$ factor through $B''[1]\to K_T[2]$ with $B''\in \B$. 
    Indeed, we have $T^*_S(C_M), T^*_S(N)\in \B'$ by Lemma \ref{lem:tred}. Since $(\T',\F')$ is effaceable in $\W$, by Proposition~3.3, the morphism $T^*_S(C_M)\to T^*_S(N)[2]$ factors as $T^*_S(C_M)\to C''\to T^*_S(N)[2]$ with $C''\in \B'$. By Proposition \ref{prop:red}, the morphism $C_M\to N[2]$ factors as $C_M\to C'[1]\to N[2]$ with $C'\in \B$. 
	By Lemma \ref{lem:conM}, there exists an object $M\in \B$ such that $M\to K_T[2]$ factor through $B''[1]\to K_T[2]$ implies that $F[1]\to T[2]$ factors though $C[1]\to T[2]$ with $C\in \B$.   
\end{proof}

\begin{theorem}\label{thm:mainr}
	Let $\Lambda$ be a finite-dimensional hereditary algebra. Assume $\T$ is a torsion class in $\mod \Lambda$ and the aisle $\U_\T$ is closed under the Serre functor. Then the torsion pair $(\T,\F)$ is effaceable.
\end{theorem}
\begin{proof}
	If $\T$ is finitely generated, then $(\T,\F)$ is effaceable by Lemma \ref{lem:enough}.
	By Lemma~\ref{lem:torp}, if $\T$ has no Ext-projective objects, then $(\T,\F)$ is effaceable.
	
	If $\T$ is not finitely generated and has non-zero Ext-projective objects, then the aisle $\U_\T$ has only finitely many indecomposable Ext-projectives \cite[Theorem 2.3]{AJS} and so is $\T$. We take an indecomposable Ext-projective in $E$ in $\T$, and $S$ is the simple top of $E$ in $\B$. By Lemma \ref{lem:tor}, $\T\cap \W$ is a torsion class in $\W$ and has one fewer Ext-projectives.   Moreover, $\U_{\T'}$ in $\db{\W}$ has one fewer Ext-projective objects. 
	
	By Proposition \ref{prop:real}, the aisle $\U_{\T'}$ has image $\U_\T\cap \ps$ under the realization functor $\Phi\colon \db{\W}\to \ps$, which is an aisle of $\ps$. By Proposition \ref{prop:main}, if $(\T',\F')$ is effaceable, then $(\T,\F)$ is effaceable. By induction,  there is an exceptional sequence $\{S_0,\ldots, S_l\}$ such that $\U_{\T}\cap  \lsup{\perp}\langle S_0,\ldots, S_l\rangle$ has no Ext-projectives.  By Lemma \ref{lem:torp}, the corresponding torsion pair is effaceable.  Applying Proposition \ref{prop:main} $l$-times, we have that $(\T,\F)$ is effaceable.
\end{proof}

\begin{proof}[Proof of Theorem \ref{thm:main}]
``(1)$\Rightarrow$ (2)" This is Theorem \ref{thm:mainr}.

``(2)$\Rightarrow$ (1)"  This is  case of Theorem \ref{thm:suffr} for $\A$ being module category of an hereditary algebra $\Lambda$.
\end{proof}
\begin{remark}
	In \cite{N}, Neeman considered the problem when a long exact sequence is obtained from the cohomology of a triangle in a derived category.
	Given a 5-term exact sequence 
	\[0\to X_0\to X_1\to X_2\to X_3\to X_4\to 0  \] in abelian category $\A$, this sequence defines a class in $\Ext^3(X_4,X_0)$. Neeman showed that the long exact sequence is obtained from taking the cohomology of a triangle in $\db{\A}$ if and only if the equivalent class in $\Ext^3(X_4,X_0)$ given by the 5-term exact sequence under the realization functor vanishes. It looks interesting to find the relationship between the Serre functor and the effaceable property of the torsion class in a more general case.
\end{remark}



\vskip 5pt

\noindent {\bf Acknowledgements}\quad  Zhe Han is supported by the National Natural Science Foundation of China (No.12001164), Ping He is supported by the National Natural Science Foundation of China (No.1250011863). The authors would like to thank Xiao-Wu Chen, Yifei Cheng, and Yiyu Cheng for their discussions and comments. We special thank Yu Zhou for his inspirational discussions and suggestions.

\end{document}